\documentclass{article}%
\usepackage{amssymb}
\usepackage{amsfonts}
\usepackage{amsmath}
\usepackage{graphicx}%
\providecommand{\U}[1]{\protect\rule{.1in}{.1in}}
\newtheorem{theorem}{Theorem}

\newtheorem{definition}[theorem]{Definition}
\newtheorem{example}[theorem]{Example}

\newtheorem{lemma}[theorem]{Lemma}

\newtheorem{proposition}[theorem]{Proposition}

\newenvironment{proof}[1][Proof]{\noindent\textbf{#1.} }{\ \rule{0.5em}{0.5em}}
\numberwithin{theorem}{section}
\numberwithin{equation}{section}
\begin{document}

\title{On Word Equations Originated from Discrete Dynamical Systems Related to
Antisymmetric Cubic Maps With Some Applications}
\author{Elias ABBOUD\\\textit{The Academic Arab Institute, Faculty of Education, }\\\textit{Beit Berl College, Doar Beit Berl 44905, Israel}\\\textit{Email: eabboud@beitberl.ac.il\ }}
\maketitle
\date{}

\renewcommand*\abstractname{Abstract\hfill} {\noindent\textbf{Abstract} In
this article, we solve some word equations originated from discrete dynamical
systems related to antisymmetric cubic map. These equations emerge when we
work with primitive and greatest words. In particular, we characterize all the
cases for which $\langle\beta_{1} \overline{\beta_{1}}\rangle=\langle\beta
_{2}\overline{\beta_{2}}\rangle$ where $\beta_{1}$ and $\beta_{2}$ are the
greatest words in $\langle\langle\beta_{1}\rangle\rangle$ and $\langle
\langle\beta_{2}\rangle\rangle$ of $\mathbf{M}(n)$.}%

\medskip
\medskip\newline\noindent\textbf{MR(2010) Subject Classification}
{\small {68R15, 37B10, 05A99} }

\section{Introduction}

Words in discrete dynamical systems have been widely studied. Sun and Helmberg
\cite{H-S} presented an algorithm for recognizing maximality of a word by
introducing an extended order of words connected with unimodal maps. Chen and
Wang \cite{C-W} studied the relation between the kneading sequences of
unimodal maps and the decomposition of necklaces. Dai et al. \cite{D-L-W1},
\cite{D-L-W2} gave some combinatorial properties of the periodic kneading
sequences of quadratic maps and antisymmetric cubic maps. Lu \cite{Lu}
introduced two extended orders on the kneading sequences of antisymmetric
cubic maps and discussed the enumeration of kneading words and the
decomposition of corresponding necklaces.

Historically, applied symbolic dynamics was first developed for unimodal maps,
which are realized by quadratic maps from the unit interval into itself. One
of the major topics which was studied by many authors is the MSS sequences,
see \cite{Abb}-\cite{D-L-W1}. On the contrary, combinatorics of symbolic
dynamics of antisymmetric cubic maps is less investigated (see \cite{D-L-W2}
and \cite{Lu}).

Word equations are very important when dealing with combinatorics of words.
Lothaire \cite[p.162]{Loth} devotes a full chapter on this topic, especially
he opens the chapter by presenting the full solution of the equation: $XY=YX.$
In this case both words $X$ and $Y$ must be powers of the same word.

In this article, we solve some word equations originated from discrete
dynamical systems related to antisymmetric cubic map. These equations are of
the form $ZW=\overline{W}\overline{Z},XY=\overline{Y}X$ and $XY=\overline{Y}Z$
where $X,Y,Z,W\in$ $W.$ In the first two cases we give explicit solutions
which implies in particular that some words are not primitive and in the
latter case we get that $XY$ is an alternating broken word. These equations
emerge when we work with primitive and greatest words. In particular, we
characterize all cases for which $\langle\beta_{1}\overline{\beta_{1}}%
\rangle=\langle\beta_{2}\overline{\beta_{2}}\rangle$ where $\beta_{1}$ and
$\beta_{2}$ are the greatest words in $\langle\langle\beta_{1}\rangle
\rangle\in M(n)$ and $\langle\langle\beta_{2}\rangle\rangle\in M(n),$
respectively, (see Theorems \ref{characterization} and
\ref{explicit characterization} below).

\section{Preliminaries}

We borrow the main notation and terminology from \cite{Lu}. Corresponding to
the symbolic dynamics of antisymmetric cubic maps, we shall be concerned with
the admissible sequences which are (1) all infinite sequences on $\{L,M,R\}$
(2) all finite sequences of the forms $\gamma\overline{C}$ and $\gamma
C\overline{\gamma C},$ where $\gamma$ is a word on $\{L,M,R\}$. Admissible
sequences are ordered by the parity-lexicographic order. The letters
$\{L,\overline{C},M,C,R\}$ are endowed with the order $L<\overline{C}<M<C<R$.
A finite sequence is said to be even (odd) if it contains an even (odd) number
of $M$'s. Let $w=w_{1}...w_{k}s...$ and $v=w_{1}...w_{k}t...$ be two
admissible sequences such that $s\neq t.$ The order relation of $w$ and $v$ is
defined as follows: when $w_{1}...w_{k}$ is even then $w>$ $v$ if $s>t$ and
$w<v$ if $s<t$; otherwise, when $w_{1}...w_{k}$ is odd then $w<$ $v$ if $s>t$
and $w>v$ if $s<t$.

Following \cite{Loth} for the terminology and notations for words, let $A$ be
a finite alphabet, whose elements are called letters. A word $\alpha$ on $A$
is a finite sequence of elements of $A$: $a_{1}a_{2}...a_{n}$ where $n$ is
called the\ length of $\alpha$, denoted by $\left\vert \alpha\right\vert $.
The set of all words on $A$ is denoted by $A^{\ast}$. The set $A^{\ast}$ is
equipped with a binary operation obtained by $(a_{1}a_{2}...a_{n})(b_{1}%
b_{2}...b_{m})=a_{1}a_{2}...a_{n}b_{1}b_{2}...b_{m}.$ The empty word is the
identity of this operation. Given a word $w\in A^{\ast}$ and $n\geq1$, by
$w^{n}$ we denote the word $ww...w$ ($n$ terms). A word $v\in$ $A^{\ast}$ is
said to be a left or a right factor of a word $x\in A^{\ast}$ if $x=vx_{1}$ or
$x=x_{2}v,$ respectively, where $x_{1},x_{2}\in A^{\ast}$. Denote by $v|x$
when $v$ is a left factor of $x$ and by $v\nmid x$ when $v$ is not a left
factor of $x.$

Define%

\[
\mathbf{W}_{n}=\{t_{1}t_{2}...t_{n}|t_{i}=L,M\text{ or }R\},
\]%
\[
\mathbf{W}=\underset{n\geq0}{\cup}\mathbf{W}_{n}.
\]
For a word $w=w_{1}...w_{k}\in W$, we define the complementation $\overline
{w}$ as $\overline{w}=\overline{w_{1}...w_{k}}$, where $\overline
{R}=L,\overline{L}=R$ and $\overline{M}=M$. A word $w\in W$ of length $n$ is
said to be primitive provided its smallest subperiod is also of length $n,$
i.e., $w\neq u^{l},l\geq2$. \ For $w\in W$, we denote by $\langle w\rangle$
the set of words which can be obtained from $w$ by cyclic permutations, and
call $\langle w\rangle$ a conjugate class or a necklace. Let $\langle\langle
w\rangle\rangle=\langle w\rangle\cup\langle\overline{w}\rangle.$ We call a
necklace $\langle w\rangle$ self-complementary if $\overline{w}\in\langle
w\rangle.$ It is clear that $\langle w\rangle$ is self-complementary if and
only if $\langle\langle w\rangle\rangle=\langle w\rangle$.

Define also the special sets of words:

{\small
\[
M(n)=\left\{  \langle\langle w\rangle\rangle=\langle w\rangle\cup
\langle\overline{w}\rangle|w\in\mathbf{W}_{n}\text{ is primitive}\right\}  ,
\]
}

{\small
\[
U(n)=\left\{  \langle w\rangle|w\in\mathbf{W}_{n}\text{ is primitive and
self-complementary}\right\}  .
\]
}

Lu \cite{Lu}, introduced an extended parity-lexicographic order on the words
in $W$. Since there are two critical points $C$ and $\overline{C}$, he
introduced two such orders, defined as follows: Let $u$ and $v$ be words in
$W$. If $u\nmid w$ and $w\nmid u$ then $u$ and $w$ are ordered according to
the parity-lexicographic order. However, if $u|w,$ say $w=uv$ then

a. The $\overline{C}$-order: $uv>u$ if $u$ is odd, otherwise $u>uv$.

b. The $C$-order: $uv>u$ if $u$ is even, otherwise $u>uv$.

If $D$ denotes $\overline{C}$ or $C$ then given $w\in W$, we call $w$ a
$D-$lexical word if $w$ is greater than all of its right shifts (right
factors) and $\overline{w}$ is less than all of its right shifts in $D-$order.

Let $L_{1}(D)=\{M,R\}$, and for $n>1$, let

{\small
\[
L_{n}(D)=\{w|w\in W_{n}\text{ and }w\text{ is }D-\text{lexical}\}.
\]
}

\begin{definition}
A word $w\in W$ is called a \underline{broken alternating word } if it is of
the form $w=(w_{1}\overline{w_{1}})^{n}w_{0}$ where $n$ is a positive integer
and $w_{0}$ is a left factor of $w_{1}\overline{w_{1}}.$
\end{definition}

We need the following elementary properties (see Lemmas 2.1.2.2 and 2.3 in [Lu]);

\begin{lemma}
\label{elementary properties}(a) Let $w=w_{1}w_{2}...w_{n}\in L_{n}(C),n>1.$
Then $w_{1}=R$ and $w_{n}=M$ or $R.$

(b) Let $\alpha,\beta,\gamma$ and $w$ be words in $W.$ If $\alpha
\beta<w<\alpha\gamma$ in $\overline{C}-$order or $C-$order, then $\alpha|w.$

(c) Suppose that $\alpha$ and $\beta$ are words in $W$ such that
$\alpha\dagger\beta$ and $\beta\dagger\alpha.$ Then in $\overline{C}-$order or
$C-$order we have;
\[
\alpha>\beta\Leftrightarrow\overline{\beta}>\overline{\alpha}%
\]

\end{lemma}

\section{Word equations on $W$}

The aim of this section is to solve the word equations $ZW=\overline
{W}\overline{Z},XY=\overline{Y}X$ and $XY=\overline{Y}Z$ where $X,Y,Z,W\in$
$W.$ We shall assume that all words are not empty.

\subsection{The word equation $ZW=\overline{W}\overline{Z}$}

\begin{proposition}
\label{ZW=WZ}Suppose $ZW=\overline{W}\overline{Z}$ where $\left\vert
ZW\right\vert =m$ and $\left\vert W\right\vert =r.$ Let $d=(m,r)$ be the
greatest common divisor. Then one of the following occurs:

(1) $\left\vert Z\right\vert =\left\vert W\right\vert $ implies $Z=\overline
{W}.$

(2) $\left\vert Z\right\vert \neq\left\vert W\right\vert $ implies (a)
$Z=(\overline{E}E)^{(\frac{m-r}{d}-1)/2}\overline{E}$ and $W=(E\overline
{E})^{(\frac{r}{d}-1)/2}E,$ if $\frac{m}{d}$ is even, where $\overline{E}$ is
a left factor of $Z$ of length $d.$ Hence, $Z$ and $W$ are alternating broken words.

Or (b) $ZW=M^{m}$ if $\frac{m}{d}$ is odd. \ \ 
\end{proposition}

\begin{proof}
(1) Obvious.

(2) Suppose $\left\vert ZW\right\vert =m$ and $\left\vert W\right\vert =r,$
then $r<m.$ Let $Z=$ $\overline{y_{1}y_{2}...y_{m-r}}$ and $W=\overline
{y_{m-r+1}...y_{m}}.$

It follows that
\begin{equation}
ZW=\overline{y_{1}}\overline{y_{2}}...\overline{y_{m}}=y_{m-r+1}...y_{m}%
y_{1}y_{2}...y_{m-r}. \label{-1}%
\end{equation}
As a result, we get the relation;%
\[
\overline{y_{i}}=y_{j},1\leq i\leq m,
\]
where
\[
m-r+i\equiv j(\operatorname{mod}m),1\leq j\leq m.
\]

Denote $t=m-r$ then obviously $d=(m,r)|t$.

\underline{Subcase 1}: $\frac{m}{d}$ is odd.

Clearly we have;
\[
\overline{y_{i}}=y_{i+t}=\overline{y_{i+2t}}=...=y_{i+(2k+1)t},
\]
where all indices are taken $\operatorname{mod}m.$ Now we ask whether there
exists a $k$ for which
\[
i+(2k+1)t\equiv i(\operatorname{mod}m).
\]
This congruence is equivalent to the following:%
\[
(2k+1)t\equiv0(\operatorname{mod}m).
\]
Thus it is sufficient to take $2k+1=\frac{m}{d},$ so that
\[
(2k+1)t\equiv\frac{m}{d}t\equiv m\frac{t}{d}\equiv0(\operatorname{mod}m).
\]
This computation can be carried out for each $i$ and therefore $\overline
{y_{i}}=y_{i}$ for every $i,1\leq i\leq m.$ Hence, $y_{i}=M,$ and $ZW=M^{m}$.

\underline{Subcase 2}: $\frac{m}{d}$ is even.

In this case the following is satisfied \ for every $1\leq i\leq m$;%

\begin{equation}
\overline{y_{i}}=y_{i+t}=\overline{y_{i+2t}}=...=\overline{y_{i+\frac{m}{d}t}%
}.\label{-3}%
\end{equation}
Clearly, the group $Z_{m}$ acts on the set $S=\{\overline{y_{1}}%
,\overline{y_{2}},...,\overline{y_{m}},y_{1},y_{2},...,y_{m}\}$ in the
following manner: for $a\in Z_{m}$ define $\overline{y_{i}}^{a}=y_{i+at}$ and
$y_{i}^{a}=\overline{y_{i+at}}.$ Then $\overline{y_{i}}^{a}=$ $\overline
{y_{i}}^{b}\Leftrightarrow y_{t+i+at}=y_{t+i+bt}$ $\Leftrightarrow
t+i+at\equiv t+i+bt(\operatorname{mod}m).$ Hence, $at\equiv
bt(\operatorname{mod}m)$ and since $t=m-r$ we have $ar\equiv
br(\operatorname{mod}m)$ which is equivalent to $a\frac{r}{d}\equiv b\frac
{r}{d}(\operatorname{mod}\frac{m}{d}).$ But $(\frac{r}{d},\frac{m}{d})=1$ so
that $a\equiv b(\operatorname{mod}\frac{m}{d}).$ This proves that the length
of each orbit is exactly $\frac{m}{d}$ and each orbit which begins with
$\overline{y_{i}}$ is given in \ref{-3}, namely; $(\overline{y_{i}}%
,y_{i+t},\overline{y_{i+2t}},...,y_{i+(\frac{m}{d}-1)t}).$ On the other hand,
if $b=0$ then $\overline{y_{i}}^{a}=$ $\overline{y_{i}}^{b}=y_{t+i}%
=\overline{y_{i}}$ and this true iff $a\equiv0(\operatorname{mod}\frac{m}%
{d}).$ Consequently, fixed points occur only under the action of the subset
$\{0,\frac{m}{d},2\frac{m}{d},...,(d-1)\frac{m}{d}\}\subseteq Z_{m}$ and each
has exactly $m$ fixed points. By the Burnside lemma the number of orbits,
which begins with $\overline{y_{i}},$ of the action of $Z_{m}$ on $S$ is:%
\begin{align*}
\#\text{orbits} &  =\frac{1}{\left\vert \mathbb{Z}_{m}\right\vert
}\underset{a\in\mathbb{Z}_{m}}{%
{\displaystyle\sum}
}\left\vert Fix(a)\right\vert \\
&  =\frac{1}{m}[\left\vert Fix(0)\right\vert +\left\vert Fix(\frac{m}%
{d})\right\vert +...+\left\vert Fix((d-1)\frac{m}{d})\right\vert ]\\
&  =\frac{1}{m}dm=d.
\end{align*}
As a result we can arrange the distinct orbits which begin with $\overline
{y_{i}}$ as follows:%
\[%
\begin{array}
[c]{ccccc}%
\overline{y_{1}}= & y_{1+t}= & \overline{y_{1+2t}}= & ... & =y_{1+(\frac{m}%
{d}-1)t}\\
\overline{y_{2}}= & y_{2+t}= & \overline{y_{2+2t}}= & ... & =y_{2+(\frac{m}%
{d}-1)t}\\
\vdots &  &  &  & \\
\overline{y_{d}}= & y_{d+t}= & \overline{y_{d+2t}}= & ... & =y_{d+(\frac{m}%
{d}-1)t}%
\end{array}
.
\]
Letting $E=$ $y_{1}y_{2}...y_{d}$ then $\overline{ZW}=y_{1}y_{2}...y_{m}$ can
be partitioned into $\frac{m}{d}$ (even) alternating orbits of the form
$(E\overline{E})^{\frac{m}{2d}}$ and $Z=$ $\overline{y_{1}y_{2}...y_{m-r}}$
can be partitioned into $\frac{m-r}{d}$ (note that $(\frac{r}{d},\frac{m}%
{d})=1$ and $\frac{m}{d}$ is even hence $\frac{r}{d}$ and $\frac{m-r}{d}$ are
odds) alternating orbits of the form $(E\overline{E})^{(\frac{m-r}{d}-1)/2}E.$
Therefore;%
\[
\overline{ZW}=y_{1}y_{2}...y_{m}=(E\overline{E})^{\frac{m}{2d}}=(E\overline
{E})^{n_{1}}%
\]

and
\[
Z=\overline{y_{1}}\overline{y_{2}}...\overline{y_{m-r}}=(\overline
{E}E)^{(\frac{m-r}{d}-1)/2}\overline{E}=(\overline{E}E)^{n_{2}}\overline{E}%
\]
where $n_{2}=(\frac{m-r}{d}-1)/2.$

Hence,
\[
W=(E\overline{E})^{(\frac{r}{d}-1)/2}E,
\]
and the result follows.
\end{proof}

\begin{example}
($m=6,r=4$)

In this case $d=(m,r)=2$ and $\frac{m}{d}=3$ is odd. Equation (\ref{-1})
becomes
\[
\overline{y_{1}}\overline{y_{2}}\overline{y_{3}}\overline{y_{4}}%
\overline{y_{5}}\overline{y_{6}}=y_{3}y_{4}y_{5}y_{6}y_{1}y_{2}.
\]
Therefore, $\overline{y_{1}}=y_{3}=\overline{y_{5}}=y_{1}.$ Likewise,
$\overline{y_{2}}=y_{4}=\overline{y_{6}}=y_{2}.$ Continuing this manner one
can easily see that $y_{i}=M$ for every $1\leq i\leq6,$ which is consistent
with the conclusion of the previous proposition.
\end{example}

\begin{example}
($m=8,r=2$)

In this case $d=(m,r)=2$ and $\frac{m}{d}=4$ is even. Equation (\ref{-1}) is
equivalent to:
\[
\overline{y_{1}}\overline{y_{2}}\overline{y_{3}}\overline{y_{4}}%
\overline{y_{5}}\overline{y_{6}}\overline{y_{7}}\overline{y_{8}}=y_{7}%
y_{8}y_{1}y_{2}y_{3}y_{4}y_{5}y_{6}.
\]
Hence,%
\begin{align*}
\overline{y_{1}} &  =y_{7}=\overline{y_{5}}=y_{3}=\overline{y_{1}},\\
\overline{y_{2}} &  =y_{8}=\overline{y_{6}}=y_{4}=\overline{y_{2}}.
\end{align*}
Letting $E=$ $y_{1}y_{2}$ then $\overline{ZW}=(y_{1}y_{2})(y_{3}y_{4}%
)(y_{5}y_{6})(y_{7}y_{8})=(E\overline{E})^{2},$ and
\begin{align*}
Z &  =\overline{y_{1}}\overline{y_{2}}\overline{y_{3}}\overline{y_{4}%
}\overline{y_{5}}\overline{y_{6}}=\overline{E}E\overline{E}\\
W &  =\overline{y_{7}}\overline{y_{8}}=E,
\end{align*}
which is consistent with the conclusion of the previous proposition since
$\frac{m}{2d}=2$ and $\frac{r}{d}=1.$
\end{example}

\subsection{The word equation $XY=\overline{Y}X$}

\begin{proposition}
\label{Eq XY=Y-X}Suppose $XY=\overline{Y}X$ then the words $XY\overline
{XY},\overline{Y}XY\overline{X}$ are not primitive.
\end{proposition}

\begin{proof}
Consider the following cases:

(1) If $\left\vert X\right\vert =\left\vert Y\right\vert =l$ then
$X=\overline{Y}$ and $Y=X$ so that $X=\overline{X}$ and $Y=\overline{Y}.$
Equivalently, $X=Y=M^{l}.$ Hence, all words on $X$ and $Y$ which contain at
least two letters are non-primitive.

(2) $\left\vert X\right\vert =l>\left\vert Y\right\vert =m.$

Let $l=qm+r,$ $0\leq r<m,$ and denote $X=x_{1}x_{2}...x_{l}$ and $Y=y_{1}%
y_{2}...y_{m}$ where $x_{i},y_{j}\in\{L,M,R\},1\leq i\leq l,1\leq j\leq m.$
Then the word equation $XY=\overline{Y}X$ is equivalent to:
\[
x_{1}x_{2}...x_{l}y_{1}y_{2}...y_{m}=\overline{y_{1}}\overline{y_{2}%
}...\overline{y_{m}}x_{1}x_{2}...x_{l}.
\]

Consequently, comparing the letters of $X$ in both sides and taking in account
the periodicity we have;%
\[
X=(\overline{y_{1}}\overline{y_{2}}...\overline{y_{m}})^{q}\overline{y_{1}%
}\overline{y_{2}}...\overline{y_{r}}=y_{m-r+1}...y_{m}(y_{1}y_{2}...y_{m}%
)^{q}.
\]
Since $l>m$ then $\ q\geq1$\ $.$ It follows that
\[
\overline{y_{1}}\overline{y_{2}}...\overline{y_{m}}=y_{m-r+1}...y_{m}%
y_{1}y_{2}...y_{m-r}.
\]
Let $Z=$ $\overline{y_{1}y_{2}...y_{m-r}}$ and $W=\overline{y_{m-r+1}...y_{m}%
}$ then the equation which was discussed in the previous proposition is
satisfied: $ZW=\overline{W}\overline{Z},$ where $\left\vert ZW\right\vert =m$
and $\left\vert W\right\vert =r.$ Let $d=(m,r)$ then we have two cases:

If $\frac{m}{d}$ is odd then $Y=\overline{ZW}=M^{m}$ and $X=M^{l}.$ In
particular, the word $XY\overline{XY}$ is not primitive.

If $\frac{m}{d}$ is even then $Y=\overline{ZW}=(E\overline{E})^{\frac{m}{2d}%
}=(E\overline{E})^{n_{1}},n_{1}\geq1.$

and
\[
X=(\overline{y_{1}}\overline{y_{2}}...\overline{y_{m}})^{q}\overline{y_{1}%
}\overline{y_{2}}...\overline{y_{r}}=(\overline{E}E)^{\frac{m}{2d}q}%
(\overline{E}E)^{(\frac{r}{d}-1)/2}\overline{E}=(\overline{E}E)^{n_{2}%
}\overline{E}%
\]
where $n_{2}=\frac{m}{2d}q+(\frac{r}{d}-1)/2\geq1.$

In particular, the words%
\begin{align*}
XY\overline{XY} &  =(\overline{E}E)^{2n_{1}+2n_{2}+1},\\
\overline{Y}XY\overline{X} &  =(\overline{E}E)^{2n_{1}+2n_{2}+1},
\end{align*}
are not primitive.

(3) $\left\vert X\right\vert =l<\left\vert Y\right\vert =m.$

The equation $XY=\overline{Y}X$ implies that $X$ is a left factor of
$\overline{Y}$ and $X$ is a right factor of $Y$. Hence, $\overline{Y}=XY_{1}$
and $Y=Y_{2}X.$ Substituting back in the equation we get $Y_{1}=Y_{2}$ and
therefore $\overline{Y}=XY_{1}$ and $Y=Y_{1}X.$ Equivalently, $\overline
{XY_{1}}=Y_{1}X.$ Apply Proposition \ref{ZW=WZ} by taking $Z=Y_{1},W=X$ then
$Y=Y_{1}X=ZW,\left\vert Y\right\vert =\left\vert ZW\right\vert =m$ and
$\left\vert X\right\vert =\left\vert W\right\vert =l.$ Let $d=(m,l).$ If
$\frac{m}{d}$ is odd then $Y=ZW=M^{m}$ and $X=W=M^{l}$ and therefore
$XY\overline{XY}$ and $\overline{Y}XY\overline{X}$ are not primitive. If
$\frac{m}{d}$ is even then $Z=(\overline{E}E)^{(\frac{m-l}{d}-1)/2}%
\overline{E}$ and $X=W=(E\overline{E})^{(\frac{l}{d}-1)/2}E$ so that
$Y=ZW=(\overline{E}E)^{n_{1}},n_{1}\geq1,XY=(E\overline{E})^{n_{2}}E,n_{2}%
\geq1,$ and $XY\overline{XY}=(E\overline{E})^{n_{2}}E(\overline{E}E)^{n_{2}%
}\overline{E}=(E\overline{E})^{2n_{2}+1}$ is not primitive.

The word $\overline{Y}XY\overline{X}$ is also not primitive. Since,
$\overline{Y}X=(E\overline{E})^{n_{1}}(E\overline{E})^{(\frac{l}{d}%
-1)/2}E=(E\overline{E})^{n_{3}}E,n_{3}\geq1$ and $\overline{Y}XY\overline
{X}=(E\overline{E})^{n_{3}}E(\overline{E}E)^{n_{3}}\overline{E}=(E\overline
{E})^{2n_{3}+1}.$
\end{proof}

\subsection{The word equation $XY=\overline{Y}Z$}

\begin{proposition}
\label{XY=Y-Z}Suppose $XY=\overline{Y}Z$ then;

(1) $\left\vert X\right\vert =\left\vert Y\right\vert =\left\vert Z\right\vert
$ implies $X=\overline{Z}$ and the equation becomes $XY=\overline{Y}%
\overline{X}.$

(2) If $\left\vert X\right\vert =\left\vert Z\right\vert <\left\vert
Y\right\vert $ then denote $\left\vert Y\right\vert =q\left\vert X\right\vert
+r.$ We have two cases: (a) for even $q,$ it follows that $\overline{X}%
=Y_{0}X_{0}$\thinspace$,$ $Z=\overline{X_{0}}Y_{0}$ and $Y=(Y_{0}%
X_{0}\overline{Y_{0}X_{0}})^{\frac{q}{2}}Y_{0}$ is a broken alternating word.
(b) for odd $q,$ it follows that $X=Y_{0}X_{0},Z=\overline{X_{0}}Y_{0}$ and
$Y=(\overline{Y_{0}X_{0}}Y_{0}X_{0})^{\frac{q-1}{2}}\overline{Y_{0}X_{0}}%
Y_{0}$ is a broken alternating word.

(3) $\left\vert X\right\vert =\left\vert Z\right\vert >\left\vert Y\right\vert
$ implies $X=\overline{Y}X_{1}$ and $Z=X_{1}Y.$
\end{proposition}

\begin{proof}
There are three cases, namely;

(1) $\left\vert X\right\vert =\left\vert Y\right\vert =\left\vert Z\right\vert
.$ In this case $X=\overline{Y}$ and $Y=Z$ hence, $X=\overline{Z}$ and the
equation becomes $XY=\overline{Y}\overline{X}$. The conclusion of Proposition
\ref{ZW=WZ} applies.

(2) $\left\vert X\right\vert =\left\vert Z\right\vert <\left\vert Y\right\vert
.$ Suppose $\left\vert Y\right\vert =q\left\vert X\right\vert +r$ and
partition $Y$ into subwords as follows: $Y=Y_{1}Y_{2}...Y_{q}Y_{0}$ where,
$\left\vert Y_{i}\right\vert =$ $\left\vert X\right\vert ,1\leq i\leq q$ and
$\left\vert Y_{0}\right\vert =r.$ Hence, the equation $XY=\overline{Y}Z$ is
equivalent to the equation:%
\[
XY_{1}Y_{2}...Y_{q}Y_{0}=\overline{Y_{1}Y_{2}...Y_{q}Y_{0}}Z.
\]
Consequently, $X=\overline{Y_{1}},Y_{1}=\overline{Y_{2}},...,Y_{q-1}%
=\overline{Y_{q}}$ and $Y_{q}Y_{0}=\overline{Y_{0}}Z.$ Moreover, $Y_{q}=X$ or
$Y_{q}=\overline{X}$ according to whether $q$ is even or odd, respectively.
Thus we have:%
\[
Y=\left\{
\begin{array}
[c]{c}%
(\overline{X}X)^{\frac{q}{2}}Y_{0},\,\ \ \ q\text{ is even and }Y_{0}\text{ is
a left factor of }\overline{Y_{q}}=\overline{X}\\
(\overline{X}X)^{\frac{q-1}{2}}\overline{X}Y_{0},\,\ \ \ q\text{ is odd and
}Y_{0}\text{ is a left factor of }\overline{Y_{q}}=X
\end{array}
\right.  .
\]
Consider two subcases;

\underline{Subcase 1}: If $q$ is even, write $\overline{X}=Y_{0}X_{0}$ then
$XY_{0}=\overline{Y_{0}X_{0}}Y_{0}=\overline{Y_{0}}Z$ implies $Z=\overline
{X_{0}}Y_{0}$ and $Y=(Y_{0}X_{0}\overline{Y_{0}X_{0}})^{\frac{q}{2}}Y_{0}$

\underline{Subcase 2}: If $q$ is odd, write $X=Y_{0}X_{0}$ then $\overline
{X}Y_{0}=\overline{Y_{0}X_{0}}Y_{0}=\overline{Y_{0}}Z$ implies $Z=\overline
{X_{0}}Y_{0}$ and $Y=(\overline{Y_{0}X_{0}}Y_{0}X_{0})^{\frac{q-1}{2}%
}\overline{Y_{0}X_{0}}Y_{0}.$

(3) $\left\vert X\right\vert =\left\vert Z\right\vert >\left\vert Y\right\vert
.$ The best possible we can get from the equation $XY=\overline{Y}Z$ \ is the
following: $X=\overline{Y}X_{1}$ and $Z=Z_{1}Y.$ It is easily seen that
$X_{1}=Z_{1}$ and therefore $X=\overline{Y}X_{1}$ and $Z=X_{1}Y.$
\end{proof}

\section{Applications}

Firstly, we prove that $\beta\overline{\beta}$ is a primitive word provided
$\beta$ is primitive and is the greatest word in $\langle\langle\beta
\rangle\rangle.$

\begin{lemma}
If $\beta\in L_{n}(C)$ or $\beta\in L_{n}(\overline{C})$ is primitive and is
the greatest word in $\langle\langle\beta\rangle\rangle$, then $\beta
\overline{\beta}$ is primitive.
\end{lemma}

\begin{proof}
Suppose that $\beta\overline{\beta}$ is not primitive. Then $\beta
\overline{\beta}=(\gamma\overline{\gamma})^{k},k>1.$ We have two cases;

Case 1. $k=2m,m\geq1.$

In this case, $\beta\overline{\beta}=(\gamma\overline{\gamma})^{2m}$ implies
that $\beta=(\gamma\overline{\gamma})^{m}=\overline{\beta}$. This implies that
$\beta=M^{l},l=2\left\vert \gamma\right\vert m\geq2,$ and since $\beta$ is
primitive we get a contradiction.

Case 2. $k=2m+1,m\geq1.$

In this case, $\beta\overline{\beta}=(\gamma\overline{\gamma})^{2m+1}$ implies
that $\beta=(\gamma\overline{\gamma})^{m}\gamma,m\geq1.$ Consider the
following subcases:

Subcase I. $\gamma=L\gamma_{1}$. Let $\beta^{\prime}=\overline{\gamma}%
\gamma(\gamma\overline{\gamma})^{m-1}\gamma\in\langle\beta\rangle$ then
$\beta^{\prime}>\beta.$ This is true since $\beta$ begins with $L$ and
$\beta^{\prime}$ begins with $R.$ Thus we have a contradiction.

Subcase II. $\gamma=R\gamma_{1}$ and $\gamma$ is even. Let $\beta
^{\prime\prime}=\gamma(\gamma\overline{\gamma})^{m}\in\langle\beta\rangle$
then $\beta^{\prime\prime}>\beta$. This is true since $\gamma L$ is a left
factor of $\beta,$ while $\gamma R$ is a left factor of $\beta^{\prime\prime
}.$ Thus we have a contradiction.

Subcase III. $\gamma=R\gamma_{1}$ and $\gamma$ is odd.

Now, suppose that $\beta\in L_{n}(C).$ By the definition of the $C-$order and
since $\gamma$ is odd we have:%
\[
\beta=(\gamma\overline{\gamma})^{m}\gamma<\gamma.
\]
But this contradicts that $\beta$ is $C-$lexical, which asserts that $\beta$
is greater than all of its right shifts. If $\beta\in L_{n}(\overline{C}),$
then
\[
\overline{\beta}=(\overline{\gamma}\gamma)^{m}\overline{\gamma}.
\]
Since $\overline{\gamma}$ is odd, then in the $\overline{C}-$order we have
$\overline{\beta}>\overline{\gamma}.$ But this contradicts the lexicality of
$\beta\in L_{n}(\overline{C}),$ which asserts that $\beta$ is greater than all
of its right shifts and $\overline{\beta}$ is less than all of its right shifts.

The result follows.
\end{proof}

\begin{lemma}
\label{primitive and greatest implies primitive}Let $\beta\in W_{n}.$ If
$\beta$ is primitive and is the greatest word in $\langle\langle\beta
\rangle\rangle$ then $\beta\overline{\beta}$ is primitive.
\end{lemma}

\begin{proof}
If $\beta\in L_{n}(C)$ or $\beta\in L_{n}(\overline{C})$ then by the previous
lemma the result follows. Suppose $\beta\notin L_{n}(C)$ and $\beta\notin
L_{n}(\overline{C}).$ Hence, by Theorem 2.8 [Lu], $\beta=\delta\overline
{\delta}$ for odd $\delta.$ On the other hand, by the discussion in the
previous lemma there is a positive integer $m$ such that $\beta\overline
{\beta}=(\gamma\overline{\gamma})^{2m+1},$ for which $\gamma$ is odd and this
implies that
\[
\beta=(\gamma\overline{\gamma})^{m}\gamma,m\geq1.
\]
Consider the following cases:

Case I. $m=2t,t\geq1.$

In this case,
\[
\beta=(\gamma\overline{\gamma})^{t}\gamma(\overline{\gamma}\gamma)^{t}%
=\delta\overline{\delta}.
\]
This equation yields that $\left\vert \gamma\right\vert $ is even$.$ Thus we
may write $\gamma=\lambda_{1}\lambda_{2},$ $\left\vert \lambda_{1}\right\vert
=\left\vert \lambda_{2}\right\vert .$ Hence, $\delta=(\gamma\overline{\gamma
})^{t}\lambda_{1}$ and $\overline{\delta}=\lambda_{2}(\overline{\gamma}%
\gamma)^{t}.$ Therefore, $\lambda_{1}=\overline{\lambda_{2}}$ and hence
$\gamma=\lambda_{1}\overline{\lambda_{1}}.$ But this contradicts the fact that
$\gamma$ is an odd word.

Case II. $m=2t+1,t\geq0.$

In this case,
\[
\beta=(\gamma\overline{\gamma})^{t}\gamma\overline{\gamma}\gamma
(\overline{\gamma}\gamma)^{t}=\delta\overline{\delta}.
\]
Letting $\gamma=\lambda_{1}\lambda_{2},$ $\left\vert \lambda_{1}\right\vert
=\left\vert \lambda_{2}\right\vert ,$ we get $\delta=(\gamma\overline{\gamma
})^{t}\lambda_{1}\lambda_{2}\overline{\lambda_{1}}$ and $\overline{\delta
}=\overline{\lambda_{2}}\lambda_{1}\lambda_{2}(\overline{\gamma}\gamma)^{t}$
which implies that $\lambda_{1}=\overline{\lambda_{2}}$ and $\lambda
_{1}=\lambda_{2}$. But then $\lambda_{1}=\lambda_{2}=M$ and $\beta$ is not
primitive, a contradiction.
\end{proof}

The next theorem characterizes all cases for which $\langle\beta_{1}%
\overline{\beta_{1}}\rangle=\langle\beta_{2}\overline{\beta_{2}}\rangle$ where
$\beta_{1}$ and $\beta_{2}$ are the greatest words in $\langle\langle\beta
_{1}\rangle\rangle\in M(n)$ and $\langle\langle\beta_{2}\rangle\rangle\in
M(n).$

\begin{theorem}
\label{characterization}If $\langle\beta_{1}\overline{\beta_{1}}%
\rangle=\langle\beta_{2}\overline{\beta_{2}}\rangle$ where $\beta_{1}$ and
$\beta_{2}$ are the greatest words in $\langle\langle\beta_{1}\rangle
\rangle\in M(n)$ and $\langle\langle\beta_{2}\rangle\rangle\in M(n),$
respectively, then one of the following occurs:

\qquad(1) $\langle\langle\beta_{1}\rangle\rangle=\langle\langle\beta
_{2}\rangle\rangle.$

\qquad(2) $\beta_{2}=\lambda\mu\in L_{n}(\overline{C}),\beta_{1}=\overline
{\mu}\lambda\in L_{n}(\overline{C})$ and $\lambda\mu=\overline{\mu}\eta.$

\qquad(3) $\beta_{1}=\overline{\mu}\lambda_{1}\mu\overline{\lambda_{1}}%
,\beta_{2}=\lambda_{1}\mu\overline{\lambda_{1}}\mu\in L_{n}(\overline{C})$ and
$\lambda_{1}\mu=\overline{\mu}\lambda_{1}.$

\qquad(4) $\beta_{2}=\lambda\mu_{1}\lambda\overline{\mu_{1}}\in L_{n}%
(\overline{C}),\beta_{1}=\overline{\mu_{1}}\overline{\lambda}\mu_{1}\lambda$
and $\lambda\mu_{1}=\overline{\mu_{1}}\overline{\lambda}.$

\qquad(5) $\beta_{1}=\overline{\mu_{1}}\overline{\lambda}\mu_{1}\lambda
,\beta_{2}=\lambda\mu_{1}\lambda\overline{\mu_{1}}\in L_{n}(\overline{C})$ and
$\overline{\mu_{1}}\overline{\lambda}=\lambda\eta.$

\qquad(6) $\beta_{1}=\overline{\mu}\lambda_{1}\mu\lambda_{1}\in L_{n}%
(\overline{C}),\beta_{2}=\overline{\lambda_{1}}\overline{\mu}\lambda_{1}\mu$
and $\overline{\mu}\lambda_{1}=\overline{\lambda_{1}}\overline{\mu}.$

\ \ \ \ \ \ (7) $\beta_{1}=\overline{\mu}_{1}\lambda\mu_{1}\lambda\in
L_{n}(\overline{C}),\beta_{2}=\lambda\mu_{1}\overline{\lambda}\overline{\mu
}_{1}$ and $\lambda\mu_{1}=\overline{\mu}_{1}\lambda.$
\end{theorem}

\begin{proof}
Suppose
\[
\langle\beta_{1}\overline{\beta_{1}}\rangle=\langle\beta_{2}\overline
{\beta_{2}}\rangle,
\]
then $\beta_{1}\overline{\beta_{1}}\in\langle\beta_{2}\overline{\beta_{2}%
}\rangle.$ Denote $\beta_{1}=u_{1}u_{2}...u_{n}$ and $\beta_{2}=v_{1}%
v_{2}...v_{n}.$ Now, If $\beta_{1}\neq\beta_{2}$ and $\beta_{1}\neq
\overline{\beta_{2}}$ then;
\[
\beta_{1}\overline{\beta_{1}}=u_{1}u_{2}...u_{n}\overline{u_{1}}%
\overline{u_{2}}...\overline{u_{n}}=\overline{v_{k+1}}...\overline{v_{n}}%
v_{1}v_{2}...v_{n}\overline{v_{1}}...\overline{v_{k}},
\]
or%
\[
\beta_{1}\overline{\beta_{1}}=u_{1}u_{2}...u_{n}\overline{u_{1}}%
\overline{u_{2}}...\overline{u_{n}}=v_{k+1}v_{k+2}...v_{n}\overline{v_{1}%
}\overline{v_{2}}...\overline{v_{n}}v_{1}v_{2}...v_{k},
\]

where $0<k<n.$

Let $\lambda=$ $v_{1}v_{2}...v_{k}$ and $\mu=v_{k+1}v_{k+2}...v_{n}$ then
$\beta_{1}=\overline{\mu}\lambda$ and $\beta_{2}=\lambda\mu$ or $\beta_{1}%
=\mu\overline{\lambda}$ and $\beta_{2}=\lambda\mu.$

Since both cases are similar we shall consider only the first case, namely;
$\lambda=$ $v_{1}v_{2}...v_{k}$ and $\mu=v_{k+1}v_{k+2}...v_{n},0<k<n,$ and
$\beta_{1}=\overline{\mu}\lambda$ and $\beta_{2}=\lambda\mu.$ If $\beta
_{1}|\beta_{2}$ or $\beta_{2}|$ $\beta_{1}$ then $\beta_{1}=\beta_{2},$
because both words have the same length. Hence, case (1) of the theorem is
satisfied. Assume therefore that $\beta_{1}\dagger\beta_{2}$ and $\beta
_{2}\dagger$ $\beta_{1}.$

We have the following subcases;

I. $\beta_{1}\in L_{n}(\overline{C})$ and $\beta_{2}\in L_{n}(\overline{C}).$

Suppose first that $\beta_{1}<$ $\beta_{2}$ in $\overline{C}-$order. If
$\lambda$ is even then
\[
\beta_{1}=\overline{\mu}\lambda<\beta_{2}=\lambda\mu<\lambda.
\]
But this contradicts the lexicality of $\beta_{1}\in L_{n}(\overline{C}),$
which asserts that $\beta_{1}$ is greater than all of its right shifts. If
$\lambda$ is odd then, by Lemma \ref{elementary properties} (c),
\[
\overline{\beta_{1}}=\mu\overline{\lambda}>\overline{\beta_{2}}=\overline
{\lambda}\overline{\mu}>\overline{\lambda}.
\]

This also contradicts the lexicality of $\beta_{1}\in L_{n}(\overline{C}),$
which asserts that $\beta_{1}$ is greater than all of its right shifts and
$\overline{\beta}$ is less than all of its right shifts.

On the other hand, if $\beta_{2}<$ $\beta_{1}$ in $\overline{C}-$order then;
\[
\beta_{2}=\lambda\mu<\beta_{1}=\overline{\mu}\lambda.
\]
Since $\beta_{2}=\lambda\mu$ is greatest in $\langle\langle\beta_{2}%
\rangle\rangle$ and $\overline{\mu}\overline{\lambda}\in\langle\langle
\beta_{2}\rangle\rangle$ then we have:
\[
\overline{\mu}\overline{\lambda}<\beta_{2}=\lambda\mu<\beta_{1}=\overline{\mu
}\lambda.
\]
Hence, $\overline{\mu}|\lambda\mu$ and we get the word equation;
\begin{equation}
\lambda\mu=\overline{\mu}\eta\label{1}%
\end{equation}

Thus, case (2) of the theorem is satisfied.

II. $\beta_{1}\notin L_{n}(\overline{C})$ or $\beta_{2}\notin L_{n}%
(\overline{C})$ but not both$.$

Since $\beta_{1}$ and $\beta_{2}$ are primitive and the greatest words in
$\langle\langle\beta_{1}\rangle\rangle$ and $\langle\langle\beta_{2}%
\rangle\rangle$, respectively, then by Theorem 2.8 [Lu], $\beta_{1}%
=\delta\overline{\delta}$ or $\beta_{2}=\delta\overline{\delta}$ for odd
$\delta.$

\underline{Subcase 1}: $\beta_{1}=\overline{\mu}\lambda=\delta\overline
{\delta}$ for odd $\delta$ and $\beta_{2}=\lambda\mu\in L_{n}(\overline{C}).$

(a) If $\left\vert \mu\right\vert =\left\vert \delta\right\vert $ then
$\overline{\mu}=\delta$ and $\lambda=\overline{\delta}$ and hence,
$\lambda=\mu.$ But then $\beta_{2}=\lambda^{2}$ which is a contradiction to
the fact that $\beta_{2}$ is primitive.

(b) If $\left\vert \mu\right\vert <\left\vert \delta\right\vert $ then
$\left\vert \lambda\right\vert >\left\vert \delta\right\vert $ and therefore
$\delta=\overline{\mu}\delta_{1}$ and $\lambda=\lambda_{1}\overline{\delta}.$
Substituting in the relation $\overline{\mu}\lambda=\delta\overline{\delta}$
we get $\delta_{1}=\lambda_{1}.$ Consequently; $\beta_{1}=\overline{\mu
}\lambda_{1}\mu\overline{\lambda_{1}}$ and $\beta_{2}=\lambda_{1}\mu
\overline{\lambda_{1}}\mu.$ Suppose firstly that $\beta_{1}=\overline{\mu
}\lambda_{1}\mu\overline{\lambda_{1}}<\beta_{2}=\lambda_{1}\mu\overline
{\lambda_{1}}\mu.$ Since $\beta_{1}$ is the greatest word in $\langle
\langle\beta_{1}\rangle\rangle$ and $\lambda_{1}\mu\overline{\lambda_{1}%
}\overline{\mu}\in\langle\langle\beta_{1}\rangle\rangle$ we have:
\[
\lambda_{1}\mu\overline{\lambda_{1}}\overline{\mu}<\beta_{1}=\overline{\mu
}\lambda_{1}\mu\overline{\lambda_{1}}<\beta_{2}=\lambda_{1}\mu\overline
{\lambda_{1}}\mu.
\]
Hence, $\lambda_{1}\mu|\beta_{1}=\overline{\mu}\lambda_{1}\mu\overline
{\lambda_{1}}$ and therefore $\lambda_{1}\mu$ is a left factor of $\beta_{1}$
of the same length as $\overline{\mu}\lambda_{1}.$ Thus we get the word
equation;
\begin{equation}
\lambda_{1}\mu=\overline{\mu}\lambda_{1}. \label{2}%
\end{equation}
Therefore, case (3) of the theorem is satisfied.

Suppose secondly that $\beta_{2}=\lambda_{1}\mu\overline{\lambda_{1}}\mu
<\beta_{1}=\overline{\mu}\lambda_{1}\mu\overline{\lambda_{1}}.$ Since
$\beta_{2}$ is the greatest word in $\langle\langle\beta_{2}\rangle\rangle$
and $\overline{\mu}\lambda_{1}\overline{\mu}\overline{\lambda_{1}}\in
\langle\langle\beta_{2}\rangle\rangle$ we have:%
\[
\overline{\mu}\lambda_{1}\overline{\mu}\overline{\lambda_{1}}<\beta
_{2}=\lambda_{1}\mu\overline{\lambda_{1}}\mu<\beta_{1}=\overline{\mu}%
\lambda_{1}\mu\overline{\lambda_{1}}.
\]
Hence, $\overline{\mu}\lambda_{1}|\beta_{2}=\lambda_{1}\mu\overline
{\lambda_{1}}\mu$ and once again we get the word equation (\ref{2}).

(c) If $\left\vert \mu\right\vert >\left\vert \delta\right\vert $ then
$\left\vert \lambda\right\vert <\left\vert \delta\right\vert $ and hence
$\overline{\mu}=\delta\mu_{1}$ and $\overline{\delta}=\delta_{1}\lambda$ from
which one can easily conclude that $\mu_{1}=\delta_{1},\beta_{1}=\overline
{\mu_{1}}\overline{\lambda}\mu_{1}\lambda$ and $\beta_{2}=\lambda\mu
_{1}\lambda\overline{\mu_{1}}.$ Suppose firstly that $\beta_{2}=\lambda\mu
_{1}\lambda\overline{\mu_{1}}<\beta_{1}=\overline{\mu_{1}}\overline{\lambda
}\mu_{1}\lambda$. Since $\beta_{2}$ is the greatest word in $\langle
\langle\beta_{2}\rangle\rangle$ and $\overline{\mu_{1}}\overline{\lambda}%
\mu_{1}\overline{\lambda}\in\langle\langle\beta_{2}\rangle\rangle$ we have:%
\[
\overline{\mu_{1}}\overline{\lambda}\mu_{1}\overline{\lambda}<\beta
_{2}=\lambda\mu_{1}\lambda\overline{\mu_{1}}<\beta_{1}=\overline{\mu_{1}%
}\overline{\lambda}\mu_{1}\lambda.
\]
Hence, $\overline{\mu_{1}}\overline{\lambda}|\beta_{2}=\lambda\mu_{1}%
\lambda\overline{\mu_{1}}$ and we get the word equation;%
\begin{equation}
\lambda\mu_{1}=\overline{\mu_{1}}\overline{\lambda}. \label{3}%
\end{equation}
Therefore, case (4) of the theorem is satisfied.

On the other hand, if $\beta_{1}=\overline{\mu_{1}}\overline{\lambda}\mu
_{1}\lambda<\beta_{2}=\lambda\mu_{1}\lambda\overline{\mu_{1}}$ then
\[
\lambda\overline{\mu_{1}}\overline{\lambda}\mu_{1}<\beta_{1}=\overline{\mu
_{1}}\overline{\lambda}\mu_{1}\lambda<\beta_{2}=\lambda\mu_{1}\lambda
\overline{\mu_{1}}.
\]
Hence, $\lambda|\beta_{1}=\overline{\mu_{1}}\overline{\lambda}\mu_{1}\lambda.$
Thus $\lambda$ is an initial subword of $\overline{\mu_{1}}\overline{\lambda
}.$ Therefore, we get the word equation
\begin{equation}
\overline{\mu_{1}}\overline{\lambda}=\lambda\eta, \label{4}%
\end{equation}
which is of the same type as equation (\ref{1}) and case (5) of the theorem is satisfied.

\underline{Subcase 2}: $\beta_{1}=\overline{\mu}\lambda\in$ $L_{n}%
(\overline{C})$ and $\beta_{2}=\lambda\mu=\delta\overline{\delta}$ for odd
$\delta.$

(a) If $\left\vert \mu\right\vert =\left\vert \delta\right\vert $ then
$\lambda=\delta$ and $\mu=\overline{\delta}$ and hence, $\lambda=\overline
{\mu}.$ But then $\beta_{1}=\lambda^{2}$ which is a contradiction to the fact
that $\beta_{1}$ is primitive.

(b) If $\left\vert \mu\right\vert <\left\vert \delta\right\vert $ then
$\left\vert \lambda\right\vert >\left\vert \delta\right\vert .$ Hence,
$\lambda=\delta\lambda_{1}$ and $\overline{\delta}=\delta_{1}\mu$ where
$\lambda_{1}=\delta_{1}.$ Therefore, $\beta_{1}=\overline{\mu}\lambda_{1}%
\mu\lambda_{1}$ and $\beta_{2}=\overline{\lambda_{1}}\overline{\mu}\lambda
_{1}\mu.$ Now, if $\beta_{1}<\beta_{2}$ then;%
\[
\overline{\lambda_{1}}\overline{\mu}\overline{\lambda_{1}}\mu<\beta
_{1}=\overline{\mu}\lambda_{1}\mu\lambda_{1}<\beta_{2}=\overline{\lambda_{1}%
}\overline{\mu}\lambda_{1}\mu.
\]
Hence, we get the word equation
\begin{equation}
\overline{\mu}\lambda_{1}=\overline{\lambda_{1}}\overline{\mu}, \label{5}%
\end{equation}
which is of the same type as equation (\ref{1}) and case (6) of the theorem is
satisfied. Similarly, if $\beta_{2}<\beta_{1}$ then%
\[
\overline{\mu}\lambda_{1}\mu\overline{\lambda_{1}}<\beta_{2}=\overline
{\lambda_{1}}\overline{\mu}\lambda_{1}\mu<\beta_{1}=\overline{\mu}\lambda
_{1}\mu\lambda_{1}.
\]
Therefore we get once again the same word equation \ref{5}.

(c) If $\left\vert \mu\right\vert >\left\vert \delta\right\vert $ then
$\left\vert \lambda\right\vert <\left\vert \delta\right\vert .$ Letting
$\delta=\lambda\delta_{1}$ and $\mu=\mu_{1}\overline{\delta}$, one can easily
see that $\delta_{1}=\mu_{1}$ and therefore $\beta_{1}=\overline{\mu}%
_{1}\lambda\mu_{1}\lambda$ and $\beta_{2}=\lambda\mu_{1}\overline{\lambda
}\overline{\mu}_{1}.$ If $\beta_{1}<\beta_{2}$ then;%
\[
\lambda\mu_{1}\lambda\overline{\mu}_{1}<\beta_{1}=\overline{\mu}_{1}\lambda
\mu_{1}\lambda<\beta_{2}=\lambda\mu_{1}\overline{\lambda}\overline{\mu}_{1}%
\]
and hence $\lambda\mu_{1}=\overline{\mu}_{1}\lambda$ which is a word equation
of the same type as \ref{2} and case (7) of the theorem is satisfied.

If $\beta_{2}<\beta_{1}$ then;%
\[
\overline{\mu}_{1}\lambda\mu_{1}\overline{\lambda}<\beta_{2}=\lambda\mu
_{1}\overline{\lambda}\overline{\mu}_{1}<\beta_{1}=\overline{\mu}_{1}%
\lambda\mu_{1}\lambda
\]
once again we end with the same word equation $\lambda\mu_{1}=\overline{\mu
}_{1}\lambda$.

III. $\beta_{1}\notin L_{n}(\overline{C})$ and $\beta_{2}\notin L_{n}%
(\overline{C}).$

By Theorem 2.8 [Lu], $\beta_{1}=\overline{\mu}\lambda=\delta\overline{\delta}$
and $\beta_{2}=\lambda\mu=\rho\overline{\rho}$ for odd $\delta$ and $\rho.$ We
shall prove that this case can not happen. Clearly we have $\left\vert
\delta\right\vert =\left\vert \rho\right\vert .$ Consider the following subcases:

Subcase 1: $\left\vert \lambda\right\vert =\left\vert \delta\right\vert
=\left\vert \rho\right\vert .$

We have $\overline{\mu}=\delta,\lambda=\overline{\delta}$ and $\lambda
=\rho,\mu=\overline{\rho}.$ Thus $\overline{\mu}=\delta=\overline{\lambda
}=\overline{\rho}=\mu.$ Therefore $\lambda$ and $\mu$ are powers of $M$ which
implies that $\beta_{1}$ and $\beta_{2}$ are not primitive, a contradiction.

Subcase 2: $\left\vert \lambda\right\vert <\left\vert \delta\right\vert
=\left\vert \rho\right\vert .$ Hence, $\left\vert \mu\right\vert >\left\vert
\delta\right\vert =\left\vert \rho\right\vert .$

We have $\overline{\mu}=\delta\mu_{2},\overline{\delta}=\delta_{1}\lambda$ and
$\rho=\lambda\rho_{1},\mu=\mu_{1}\overline{\rho}.$ It is easily seen that
$\mu_{2}=\delta_{1}$ and $\rho_{1}=\mu_{1}.$ Hence, we may write
$\overline{\mu}=\delta\mu_{2},\overline{\delta}=\mu_{2}\lambda$ and
$\rho=\lambda\mu_{1},\mu=\mu_{1}\overline{\rho}.$ In particular,
$\overline{\mu}=\overline{\mu_{2}\lambda}\mu_{2},$ and $\mu=\mu_{1}%
\overline{\lambda\mu_{1}}.$ Combining both these two equations we see that
$\mu=\mu_{2}\lambda\overline{\mu_{2}}=\mu_{1}\overline{\lambda\mu_{1}}.$ This
implies that $\mu_{1}=\mu_{2}$ and $\lambda=\overline{\lambda}=M^{l}%
,l=\left\vert \lambda\right\vert .$ But this contradicts the fact that the
first letter of $\beta_{2}=\lambda\mu,$ which is primitive and the greatest
word in $\langle\langle\beta_{2}\rangle\rangle,$ must be $R$ and not $M.$

Subcase 3: $\left\vert \lambda\right\vert >\left\vert \delta\right\vert
=\left\vert \rho\right\vert .$ Hence, $\left\vert \mu\right\vert <\left\vert
\delta\right\vert =\left\vert \rho\right\vert .$

Then $\delta=\overline{\mu}\delta_{1},\lambda=\lambda_{2}\overline{\delta}$
and $\lambda=\rho\lambda_{1},\overline{\rho}=\rho_{1}\mu.$ Once again it is
easily seen that $\delta_{1}=\lambda_{2}$ and $\lambda_{1}=\rho_{1}.$ Hence,
we conclude that $\lambda=\lambda_{2}\mu\overline{\lambda_{2}}=\overline
{\lambda_{1}\mu}\lambda_{1}.$ Therefore $\mu=\mu=M^{l},l=\left\vert
\mu\right\vert $ and this contradicts the fact that the first letter of
$\beta_{1}=\overline{\mu}\lambda,$ which is primitive and the greatest word in
$\langle\langle\beta_{1}\rangle\rangle,$ must be $R$ and not $M.$
\end{proof}

\begin{lemma}
If $\beta=EFE,$ where $E,F\in W$ then $\beta\notin L_{n}(\overline{C})$ and
$\beta\notin L_{n}(C).$
\end{lemma}

\begin{proof}
Both cases are similar so we prove the first one, namely; $\beta\notin
L_{n}(\overline{C}).$ Suppose the contrary, then $\beta=EFE\in L_{n}%
(\overline{C})$ implies that $\beta=EFE>E$. Hence, in $\overline{C}-$order,
$E$ must be an odd word. On the other hand, $\overline{\beta}=\overline
{EFE}<\overline{E}$ implies that, in $\overline{C}-$order, $\overline{E}$ is
even. Since $E$ and $\overline{E}$ have the same parity we get a contradiction.
\end{proof}

Now, the $\overline{C}-$lexicality of $\beta_{1}$ or $\beta_{2}$ which was
emphasized in Theorem \ref{characterization} allows us to make the following
refinement of the previous theorem.

\begin{theorem}
\label{explicit characterization}If $\langle\beta_{1}\overline{\beta_{1}%
}\rangle=\langle\beta_{2}\overline{\beta_{2}}\rangle$ where $\beta_{1}$ and
$\beta_{2}$ are the greatest words in $\langle\langle\beta_{1}\rangle
\rangle\in M(n)$ and $\langle\langle\beta_{2}\rangle\rangle\in M(n),$
respectively, and $\langle\langle\beta_{1}\rangle\rangle\neq\langle
\langle\beta_{2}\rangle\rangle$ then one of the following occurs;

(a) $\beta_{1}=\overline{\mu}\lambda=\overline{\mu}^{2}X_{1}\in L_{n}%
(\overline{C}),\beta_{2}=\lambda\mu=\overline{\mu}X_{1}\mu\in L_{n}%
(\overline{C}).$

(b) $\beta_{1}=\overline{\mu_{1}}\overline{\lambda}\mu_{1}\lambda=\lambda
X_{1}\overline{\lambda}\overline{\lambda}\overline{X_{1}}\lambda$ and
$\beta_{2}=\lambda\mu_{1}\lambda\overline{\mu_{1}}=\lambda\overline{\lambda
}\overline{X_{1}}\lambda\lambda X_{1}\in L_{n}(\overline{C}).$
\end{theorem}

\begin{proof}
By Theorem \ref{characterization}, we ought to consider the following cases:

I. $\beta_{2}=\lambda\mu\in L_{n}(\overline{C}),\beta_{1}=\overline{\mu
}\lambda\in L_{n}(\overline{C})$ and $\lambda\mu=\overline{\mu}\eta.$ Applying
the conclusion of Proposition \ref{XY=Y-Z} by substituting $X,Y$ and $Z$ as
shown in the following table (namely: $X=\lambda,Y=\mu$ and $Z=\eta$):%
\[%
\begin{array}
[c]{ccc}%
X & Y & Z\\
\lambda & \mu & \eta
\end{array}
\]
we have the following subcases:

\underline{Subcase 1}: If $\left\vert \lambda\right\vert =\left\vert
\mu\right\vert =\left\vert \eta\right\vert $ then $\lambda=\overline{\eta}$
and we get the word equation: $\lambda\mu=\overline{\mu}\overline{\lambda}.$
Applying Proposition \ref{ZW=WZ} by substituting,%
\[%
\begin{array}
[c]{cc}%
Z & W\\
\lambda & \mu
\end{array}
\]
it follows that $\lambda=\overline{\mu}$ and therefore $\beta_{1}%
=\overline{\mu}\lambda=\lambda^{2}$ is not primitive, a contradiction.

\underline{Subcase 2}: If $\left\vert \lambda\right\vert =\left\vert
\eta\right\vert <\left\vert \mu\right\vert $ then either (a) $\overline
{\lambda}=Y_{0}X_{0}$\thinspace$,$ $\eta=\overline{X_{0}}Y_{0}$ and
$\mu=(Y_{0}X_{0}\overline{Y_{0}X_{0}})^{n_{1}}Y_{0},n_{1}\geq1,$ which
implies:%
\begin{align*}
\beta_{2}  &  =\lambda\mu=\overline{Y_{0}X_{0}}(Y_{0}X_{0}\overline{Y_{0}%
X_{0}})^{n_{1}}Y_{0}=(\overline{Y_{0}X_{0}}Y_{0}X_{0})^{n_{1}}\overline
{Y_{0}X_{0}}Y_{0},\\
\beta_{1}  &  =\overline{\mu}\lambda=(\overline{Y_{0}X_{0}}Y_{0}X_{0})^{n_{1}%
}\overline{Y_{0}}\overline{Y_{0}X_{0}}.
\end{align*}
By the previous lemma this leads to a contradiction, since $\beta_{1}\in
L_{n}(\overline{C})$ have the same left and right factor $\overline{Y_{0}%
X_{0}}$, or (b) $\lambda=Y_{0}X_{0},\eta=\overline{X_{0}}Y_{0}$ and
$\mu=(\overline{Y_{0}X_{0}}Y_{0}X_{0})^{n_{2}}\overline{Y_{0}X_{0}}Y_{0}$
which implies:%
\begin{align*}
\beta_{2}  &  =\lambda\mu=Y_{0}X_{0}(\overline{Y_{0}X_{0}}Y_{0}X_{0})^{n_{2}%
}\overline{Y_{0}X_{0}}Y_{0}=(Y_{0}X_{0}\overline{Y_{0}X_{0}})^{n_{2}+1}%
Y_{0},\\
\beta_{1}  &  =\overline{\mu}\lambda=(Y_{0}X_{0}\overline{Y_{0}X_{0}})^{n_{2}%
}Y_{0}X_{0}\overline{Y_{0}}Y_{0}X_{0}=(Y_{0}X_{0}\overline{Y_{0}X_{0}}%
)^{n_{2}+1}X_{0}.
\end{align*}
But then we have a contradiction, according the previous lemma, since
$\beta_{2}\in L_{n}(\overline{C}).$

\underline{Subcase 3}: If $\left\vert \lambda\right\vert =\left\vert
\eta\right\vert >\left\vert \mu\right\vert $ then $\lambda=\overline{\mu}%
X_{1}$ and $\eta=X_{1}\mu$ which implies
\begin{align*}
\beta_{2}  &  =\overline{\mu}X_{1}\mu,\\
\beta_{1}  &  =\overline{\mu}\overline{\mu}X_{1}=\overline{\mu}^{2}X_{1}.
\end{align*}
This case is possible.

II. $\beta_{1}=\overline{\mu}\lambda_{1}\mu\overline{\lambda_{1}},\beta
_{2}=\lambda_{1}\mu\overline{\lambda_{1}}\mu$ and $\lambda_{1}\mu
=\overline{\mu}\lambda_{1}.$ Applying Proposition \ref{Eq XY=Y-X} by
substituting%
\[%
\begin{array}
[c]{cc}%
X & Y\\
\lambda_{1} & \mu
\end{array}
\]
we conclude that $\beta_{1}=\overline{\mu}\lambda_{1}\mu\overline{\lambda_{1}%
}$ is not primitive, a contradiction.

III. $\beta_{2}=\lambda\mu_{1}\lambda\overline{\mu_{1}}\in L_{n}(\overline
{C}),\beta_{1}=\overline{\mu_{1}}\overline{\lambda}\mu_{1}\lambda$ and
$\lambda\mu_{1}=\overline{\mu_{1}}\overline{\lambda}.$ If $\left\vert
\lambda\right\vert \neq\left\vert \mu_{1}\right\vert $ then applying
Proposition \ref{ZW=WZ} by the substitution%
\[%
\begin{array}
[c]{cc}%
Z & W\\
\lambda & \mu_{1}%
\end{array}
\]
we get\ that $\lambda$ and $\mu_{1}$ are powers of $M$ and therefore
$\beta_{1}$ and $\beta_{2}$ are not primitive which is a contradiction, or
$\lambda=(\overline{E}E)^{n_{1}}\overline{E},\mu_{1}=(E\overline{E})^{n_{2}}E$
and hence, $\beta_{1}=\overline{\mu_{1}}\overline{\lambda}\mu_{1}%
\lambda=(\overline{E}E)^{n_{1}+n_{2}+1}(E\overline{E})^{n_{1}+n_{2}+1}$ and
$\beta_{2}=\lambda\mu_{1}\lambda\overline{\mu_{1}}=(\overline{E}%
E)^{2n_{1}+n_{2}+1}\overline{E}^{2}(E\overline{E})^{n_{2}}.$

If $\left\vert \lambda\right\vert =\left\vert \mu_{1}\right\vert $ then
$\lambda=\overline{\mu_{1}}$ and hence;%
\begin{align*}
\beta_{2}  &  =\lambda\mu_{1}\lambda\overline{\mu_{1}}=\overline{\mu_{1}}%
\mu_{1}\overline{\mu_{1}}\overline{\mu_{1}},\\
\beta_{1}  &  =\overline{\mu_{1}}\overline{\lambda}\mu_{1}\lambda
=\overline{\mu_{1}}\mu_{1}\mu_{1}\overline{\mu_{1}}.
\end{align*}
In both cases, by the previous lemma this leads to a contradiction since
$\beta_{2}\in L_{n}(\overline{C}).$

IV. $\beta_{1}=\overline{\mu_{1}}\overline{\lambda}\mu_{1}\lambda,\beta
_{2}=\lambda\mu_{1}\lambda\overline{\mu_{1}}\in L_{n}(\overline{C})$ and
$\overline{\mu_{1}}\overline{\lambda}=\lambda\eta.$ Applying the conclusion of
Proposition \ref{XY=Y-Z} making the following substitution%
\[%
\begin{array}
[c]{ccc}%
X & Y & Z\\
\overline{\mu_{1}} & \overline{\lambda} & \eta
\end{array}
\]
we have the following subcases:

\underline{Subcase 1}: If $\left\vert \mu_{1}\right\vert =\left\vert
\lambda\right\vert =\left\vert \eta\right\vert $ then $\overline{\mu_{1}%
}=\overline{\eta}$ and we get the word equation $\lambda\mu_{1}=\overline
{\mu_{1}}\overline{\lambda}.$ Applying Proposition \ref{ZW=WZ} by substituting%
\[%
\begin{array}
[c]{cc}%
Z & W\\
\lambda & \mu_{1}%
\end{array}
\]
it follows that $\lambda=\overline{\mu_{1}}$ and therefore; $\beta
_{1}=\overline{\mu_{1}}\mu_{1}\mu_{1}\overline{\mu_{1}}$ and $\beta
_{2}=\overline{\mu_{1}}\mu_{1}\overline{\mu_{1}}\overline{\mu_{1}}$ which is
similar to a previous case and can not occur.

\underline{Subcase 2}: If $\left\vert \mu_{1}\right\vert =\left\vert
\eta\right\vert <\left\vert \lambda\right\vert $ then either $\mu_{1}%
=Y_{0}X_{0}$\thinspace$,$ $\eta=\overline{X_{0}}Y_{0}$ and $\overline{\lambda
}=(Y_{0}X_{0}\overline{Y_{0}X_{0}})^{n_{1}}Y_{0},n_{1}\geq1,$ which implies:%
\begin{align*}
\beta_{1}  &  =\overline{\mu_{1}}\overline{\lambda}\mu_{1}\lambda
=(\overline{Y_{0}X_{0}}Y_{0}X_{0})^{n_{1}}\overline{Y_{0}X_{0}}Y_{0}%
(Y_{0}X_{0}\overline{Y_{0}X_{0}})^{n_{1}}Y_{0}X_{0}\overline{Y_{0}},\\
\beta_{2}  &  =\lambda\mu_{1}\lambda\overline{\mu_{1}}=(\overline{Y_{0}X_{0}%
}Y_{0}X_{0})^{n_{1}}\overline{Y_{0}}Y_{0}X_{0}(\overline{Y_{0}X_{0}}Y_{0}%
X_{0})^{n_{1}}\overline{Y_{0}}\overline{Y_{0}X_{0}},
\end{align*}
or $\overline{\mu_{1}}=Y_{0}X_{0}$\thinspace$,$ $\eta=\overline{X_{0}}Y_{0}$
and $\overline{\lambda}=(\overline{Y_{0}X_{0}}Y_{0}X_{0})^{n_{2}}%
\overline{Y_{0}X_{0}}Y_{0}$ which implies:%
\begin{align*}
\beta_{1}  &  =\overline{\mu_{1}}\overline{\lambda}\mu_{1}\lambda=(Y_{0}%
X_{0}\overline{Y_{0}X_{0}})^{n_{2}+1}Y_{0}(\overline{Y_{0}X_{0}}Y_{0}%
X_{0})^{n_{2}+1}\overline{Y_{0}},\\
\beta_{2}  &  =\lambda\mu_{1}\lambda\overline{\mu_{1}}=(Y_{0}X_{0}%
\overline{Y_{0}X_{0}})^{n_{2}}Y_{0}X_{0}\overline{Y_{0}}(\overline{Y_{0}X_{0}%
}Y_{0}X_{0})^{n_{2}+1}\overline{Y_{0}}Y_{0}X_{0}.
\end{align*}
Both cases can not happen, by the previous lemma, since $\beta_{2}=\lambda
\mu_{1}\lambda\overline{\mu_{1}}\in L_{n}(\overline{C}).$

\underline{Subcase 3}: If $\left\vert \mu_{1}\right\vert =\left\vert
\eta\right\vert >\left\vert \lambda\right\vert $ then $\overline{\mu_{1}%
}=\lambda X_{1}$ and $\eta=X_{1}\overline{\lambda}$ which implies:%
\begin{align*}
\beta_{1}  &  =\overline{\mu_{1}}\overline{\lambda}\mu_{1}\lambda=\lambda
X_{1}\overline{\lambda}\overline{\lambda X_{1}}\lambda,\\
\beta_{2}  &  =\lambda\mu_{1}\lambda\overline{\mu_{1}}=\lambda\overline
{\lambda X_{1}}\lambda\lambda X_{1}.
\end{align*}
This case is possible.

V. $\beta_{1}=\overline{\mu}\lambda_{1}\mu\lambda_{1},\beta_{2}=\overline
{\lambda_{1}}\overline{\mu}\lambda_{1}\mu$ and $\overline{\mu}\lambda
_{1}=\overline{\lambda_{1}}\overline{\mu}.$ Making the substitution%
\[%
\begin{array}
[c]{cc}%
X & Y\\
\overline{\mu} & \lambda_{1}%
\end{array}
\]
in Proposition \ref{Eq XY=Y-X}, it follows that $\beta_{2}=\overline
{\lambda_{1}}\overline{\mu}\lambda_{1}\mu$ is not primitive, a contradiction.

VI. $\beta_{1}=\overline{\mu}_{1}\lambda\mu_{1}\lambda,\beta_{2}=\lambda
\mu_{1}\overline{\lambda}\overline{\mu}_{1}$ and $\lambda\mu_{1}=\overline
{\mu}_{1}\lambda.$ By Proposition \ref{Eq XY=Y-X}, we get that $\beta
_{2}=\lambda\mu_{1}\overline{\lambda}\overline{\mu}_{1}$ is not primitive, a contradiction.
\end{proof}

\section{Concluding remarks}

Theorems \ref{characterization} and \ref{explicit characterization} may shed
some light on a problem, posed by Lu \cite[p. 2192]{Lu}, whether there is a
bijection between the sets $M(n)$ and $U(2n)$. In fact, the relation
\[
\Psi:\mathbf{M}(n)\rightarrow\mathbf{U}(2n)
\]
defined as $\Psi(\langle\langle\beta\rangle\rangle)=\langle\beta
\overline{\beta}\rangle,$ where $\langle\langle\beta\rangle\rangle\in M(n)$
and $\beta$ is the greatest word in $\langle\langle\beta\rangle\rangle$, is a
function by Lemma \ref{primitive and greatest implies primitive}, but need not
be a bijection. The problem seems to be difficult and needs further investigation.

We conclude by some examples illustrating the last theorem. For $n=4$ and by
the computations which was carried out in \cite[p. 2189]{Lu}, we have
$\left\vert \mathbf{M}(4)\right\vert =10.$ The elements $\langle\langle
\beta\rangle\rangle$ of $M(4)$ for which $\beta$ is the greatest word in
$\langle\langle\beta\rangle\rangle$ are
\begin{align*}
&  \langle\langle R^{3}M\rangle\rangle,\langle\langle R^{3}L\rangle
\rangle,\langle\langle R^{2}M^{2}\rangle\rangle,\langle\langle R^{2}%
ML\rangle\rangle,\langle\langle R^{2}LM\rangle\rangle,\\
&  \langle\langle R^{2}L^{2}\rangle\rangle,\langle\langle RM^{3}\rangle
\rangle,\langle\langle RM^{2}L\rangle\rangle,\langle\langle RMRL\rangle
\rangle,\langle\langle RMLM\rangle\rangle.
\end{align*}
Now, if $\langle\langle\beta\rangle\rangle\in\left\{  \langle\langle
R^{3}M\rangle\rangle,\langle\langle R^{2}ML\rangle\rangle\right\}  $ then
$\langle\beta\overline{\beta}\rangle=\langle R^{3}ML^{3}M\rangle,$ if
$\langle\langle\beta\rangle\rangle\in\{\langle\langle R^{3}L\rangle
\rangle,\langle\langle R^{2}L^{2}\rangle\rangle\}$ then $\langle\beta
\overline{\beta}\rangle=\langle R^{4}L^{4}\rangle,$ and if $\langle
\langle\beta\rangle\rangle\in\{\langle\langle R^{2}M^{2}\rangle\rangle
,\langle\langle RM^{2}L\rangle\rangle\}$ then $\langle\beta\overline{\beta
}\rangle=\langle R^{2}M^{2}L^{2}M^{2}\rangle.$ Notice that all the three pairs
satisfy condition (a) of Theorem \ref{explicit characterization}.

On the other hand, if $\lambda=R^{2}M$ and $X_{1}=LM\,\ $then
\[
\beta_{1}=R^{2}MLML^{2}ML^{2}MRMR^{2}M\in\mathbf{M}(14)
\]
and
\[
\beta_{2}=R^{2}ML^{2}MRMR^{2}MR^{2}MLM\in\mathbf{M}(14)
\]
satisfy condition (b) of Theorem \ref{explicit characterization}. Notice that
$\beta_{1}$ and $\beta_{2}$ are greatest words in $\langle\langle\beta
_{1}\rangle\rangle$ and $\langle\langle\beta_{2}\rangle\rangle$, respectively,
and $\beta_{2}\in L_{n}(\overline{C}).$

\ 

\begin{flushleft}
Acknowledgement This research was supported by Beit Berl College. Thanks are
due to the referee for his precious time and comments.
\end{flushleft}

\end{document}